\documentclass{amsart}
\usepackage[dvips]{color}
\usepackage{graphicx}

\usepackage{latexsym}
\usepackage{amssymb}
\usepackage{amsmath}
\usepackage{amsthm}
\usepackage{amscd}
\theoremstyle{plain}

\newcommand{\const}{10^{(2k+4)^2+2k+3}\frac 1 {D'^{k}}\left(10^{-\sum_{j=3}^{2k+4}j^2}\right)^2}

\usepackage{amsmath,amsfonts,amssymb,amsthm,epsfig,float, color,tikz}
\usepackage[margin=1.5in]{geometry}
\newcommand{\bb}{\mathbb}

\newcommand{\C}{\bb C}

\newcommand{\R}{\bb R}
\newcommand{\N}{\bb N}

\newcommand{\om}{\omega}

\newcommand{\Hdim}{\operatorname{Hdim}}
\newcommand{\nue}{\operatorname{NUE}}

\newcommand{\spanD}{\text{span}_{\Delta}}

\newcommand{\T}{\mathcal T}

\newcommand{\hh}{\mathcal H}

\newtheorem{Theorem}{Theorem}
\newtheorem{thm}[Theorem]{Theorem}
\newtheorem{lem}[Theorem]{Lemma}
\newtheorem{defin}[Theorem]{Definition}
\newtheorem{prop}[Theorem]{Proposition}

\theoremstyle{definition}
\newtheorem{rem}[Theorem]{Remark}
\newtheorem{cor}[Theorem]{Corollary}

\newtheorem{claim}[Theorem]{Claim}

\numberwithin{Theorem}{section}

\newtheorem*{lemma*}{Lemma}
\newtheorem*{question*}{Question}

\newtheorem*{theorem*}{Theorem}

\numberwithin{equation}{section}

\begin{document}

\title[Hausdorff Dimension of Non-Ergodic IETS]{The Hausdorff Dimension of Non-Uniquely Ergodic directions in $\hh(2)$ is almost everywhere $1/2$}
\author{Jayadev S.~Athreya}
\author{Jon Chaika}

\email{jathreya@illinois.edu}
\email{chaika@math.utah.edu}
\address{Department of Mathematics, University of Illinois Urbana-Champaign, 1409 W. Green Street, Urbana, IL 61801, USA}
\address{Department of Mathematics, University of Utah, 155 S 1400 E Room 233, Salt Lake City, UT 84112}

    \thanks{J.S.A. partially supported by NSF grant DMS 1069153, NSF grants DMS 1107452, 1107263, 1107367 ``RNMS: GEometric structures And Representation varieties" (the GEAR Network), and NSF CAREER grant DMS 1351853}
    \thanks{J.C. partially supported by NSF grants DMS 1004372, 1300550}
\begin{abstract} We show that for almost every (with respect to Masur-Veech measure) $\om \in \hh(2)$, the set of angles $\theta \in [0, 2\pi)$ so that $e^{i\theta}\om$ has non-uniquely ergodic vertical foliation has Hausdorff dimension (and codimension) $1/2$.
\end{abstract}
\maketitle
\tableofcontents
\section{Introduction}

A genus $g$ \emph{translation surface} $(X, \om)$ is a compact, genus $g$ Riemann surface together with a holomorphic one-form $\om$. This gives a structure of a flat metric away from a finite number of singular points, as integrating the one-form $\omega$ gives charts (away from zeros of $\omega$) to $\C$ where the transition functions between charts are translations. The zeros of $\om$ are singular points of the metric, and have cone angles $2\pi(k+1)$ at a zero of order $k$. Translation surfaces inherit a straight line unit speed flow in each direction $\theta \in [0,2\pi)$ (corresponding to the foliation $\mbox{Re}(e^{i\theta} \om) = 0$). These flows preserve Lebesgue measure on the surface. A key result on the ergodic properties of these flows was proved by Kerckhoff, Masur and Smillie~\cite{KMS}:

\begin{theorem*}\cite[Theorem 2]{KMS} For every translation surface the flow in almost every direction is uniquely ergodic with respect to Lebesgue measure.
\end{theorem*}
Moduli spaces of translation surfaces are stratified by their genus $g$ and the combinatorics of their singularities. We say a singularity has order $k$ if the angle is $2\pi(k+1)$. The Gauss-Bonnet theorem implies that the sum of orders of singularities on a genus $g$ surface is $2g-2$. Given a partition $\alpha = (\alpha_1, \ldots, \alpha_m) \in \N^m$, $\sum \alpha_i = 2g-2$, we define the stratum $\hh = \hh(\alpha)$ to be the moduli space of (unit-area) translation surfaces with singularity pattern $\alpha$. On each stratum $\hh$, there are coordinate charts to an appropriate Euclidean space, and pulling back Lebesgue measure yields a natural measure $\mu_{MV}$, known as \emph{Masur-Veech} measure. Similarly, pulling back Euclidean distance yields a (local) metric on flat surfaces in a given stratum. For each translation surface, there is a countable set of directions where the flow is not minimal (that is, there are non-dense infinite trajectories). Moreover, by a theorem of Masur-Smillie~\cite{MS} for almost every translation surface, there is an uncountable set of non-uniquely ergodic directions. Given a translation surface $\om$, let $$\nue(\om) : = \{\theta: \mbox{ vertical flow on }e^{i\theta} \om \mbox{ is non-uniquely ergodic}\}.$$
\begin{theorem*}\label{thm:MS} \cite[Main Theorem]{MS} In every stratum of translation surfaces $\hh(\alpha)$ of surfaces of genus at least 2 there is a constant $c = c(\alpha)>0$ such that for $\mu_{MV}$-almost every flat surface $\om \in \hh$ , $$\Hdim(\nue(\om)) = c.$$\end{theorem*}
\noindent We call the constant $c=c(\alpha)$ the \emph{Masur-Smillie constant} of the stratum of $\hh(\alpha)$. Masur~\cite{Mhdim} showed that $c(\alpha) \le 1/2$ for all $\alpha$. The main result of this paper is that for $\hh(2)$, the Masur-Smillie constant is $1/2$.
\begin{thm}\label{thm:main} For $\mu_{MV}$-almost every $\om \in \hh(2)$, $$\Hdim(\nue(\om))=\frac{1}{2} = 1 - \frac 1 2 .$$\end{thm}
\noindent By our methods we also obtain the Hausdorff dimension of the set of translation surfaces in $\hh(2)$ where the vertical flow is non-uniquely ergodic. The real dimension of $\hh(2)$ is $7$, and we have: 
\begin{thm}\label{thm:stratum} $$\Hdim\left(\{\om \in \hh(2): \mbox{ vertical flow on } \om \mbox{ is non-uniquely ergodic}\}\right) = \frac{13}{2} = 7 - \frac 1 2.$$ 
\end{thm}

\begin{rem}This is the first time the Masur-Smillie constant for a stratum has been identified. Earlier, Cheung, Hubert, and Masur ~\cite{CHM} identified the Hausdorff dimension of non-uniquely ergodic directions for the historically important example of two symmetric tori glued along a slit. The Hausdorff dimension is either $\frac 1 2 $ or 0 and they gave an explicit description of these two cases based on the diophantine properties of the length of the slit. Earlier,  Cheung \cite{yit thesis} had found an example of two symmetric tori glued along a slit where the set of non-uniquely ergodic directions has Hausdorff dimension $\frac 1 2$. In the paper's appendix, Boshernitzan showed a residual set of these examples have that the set of non-uniquely ergodic directions has Hausdorff dimension $0$. All of these results deal with a measure zero subset of the stratum $\hh(1,1).$

Prior to the paper \cite{KMS}, Masur \cite{Masur1} and Veech \cite{Veech} independently showed that for almost every flat surface the flow in almost every direction was uniquely ergodic with respect to Lebesgue measure. Constructions of non-uniquely ergodic IETs are due to Sataev \cite{Sataev}, Keane~\cite{Keane2} and Keynes-Newton~\cite{KeynesNewton}, and, anachronistically, Veech~\cite{Veech69}.
\end{rem}

\noindent To prove Theorems \ref{thm:main} and \ref{thm:stratum}, we establish a related theorem for interval exchange transformations (IETs) (see Section \ref{sec:background} for the definition of IETs). Given a permutation $\pi \in S_m$ on $m$-letters, we parameterize the set of IETs which have $\pi$ as a permutation by the unit simplex $$\Delta_m : = \left\{\lambda \in \R_+^m: \sum_{i=1}^m \lambda_i = 1\right\}.$$ We denote the IET with length vector $\lambda$ and permutation $\pi$ by $T_{\lambda, \pi}$. Note that the real dimension of $\Delta_m$ is $m-1$.
\begin{thm} \label{thm:iet} The Hausdorff dimension of the set $\nue(4321)$ of non-uniquely ergodic 4-IETs of $[0,1)$ with permutation $\pi_0 = (4321)$ is $\frac 5 2$. That is $$\Hdim(\nue(4321)) = \Hdim\left(\left\{\lambda \in \Delta_4: T_{\lambda, \pi_0} \mbox{ is non-uniquely ergodic } \right\}\right) = \frac 5 2 = 3 - \frac 1 2.$$
\end{thm}
\begin{rem}Theorems \ref{thm:main}, \ref{thm:stratum} and \ref{thm:iet} all state that the \emph{Hausdorff codimension} of non-uniquely ergodic objects is $\frac 1 2$.
\end{rem}
By combining work of Masur \cite{Mhdim}, Minsky-Weiss \cite{MW} and a basic result in metric geometry \cite{mattila} we obtain the following general result:
\begin{thm}\label{ub} The set of non-uniquely ergodic $n$-IETs has Hausdorff codimension at least $\frac 12 $. The set of flat surfaces in any (connected component of any) stratum whose vertical flow is not uniquely ergodic has Hausdorff codimension at least $\frac 12$.
\end{thm}
\subsection{Outline of proof}
Masur's theorem \cite{Mhdim}, together with standard results in metric geometry, provides the upper bound. For the lower bound, we generate specific paths in Rauzy induction (see Section \ref{sec:rauzy}) that by a criterion of Veech \cite{Veech78} (Lemma \ref{not ue} in this paper) give minimal and non-uniquely ergodic IETs. The cylinder sets of these paths have nice geometric properties (Section \ref{sec:properties}). This allows us to construct a measure carried on these non-uniquely ergodic IETs that by Frostman's Lemma shows they have Hasudorff dimension $\frac 5 2$ (Sections \ref{sec:frostman}  and \ref{sec:lower}). To prove Theorem \ref{thm:stratum} we appeal to a standard decomposition of the stratum into stable and unstable foliations. The property that the vertical flow is non-uniquely ergodic depends only on the unstable coordinate and reduces the problem to IETs. To prove Theorem \ref{thm:main} we use a result of Minsky and Weiss \cite{MW} that varying directions on a fixed flat surface gives lines in IET space (technically, they and also this paper consider the horocycle through a translation surface). Having enough lines implies by standard results in metric geometry that many of the lines intersect non-uniquely ergodic IETs in Hausdorff dimension $\frac 1 2$. This establishes Theorem \ref{thm:main} for a positive measure set of flat surfaces. There is an $SL_2(\mathbb{R})$ action on each stratum which is ergodic (on connected components of strata). The Hausdorff dimension of non-uniquely ergodic directions is invariant under this action. The ergodicity of the $SL_2(\mathbb{R})$ action lets us go from positive to full measure and proves Theorem \ref{thm:main}.

\medskip
\noindent\textbf{Acknowledgements:} We would like to thank Michael Boshernitzan, Yitwah Cheung, Marianna Csornyei, Howard Masur, and Jeremy Tyson for useful discussions. J.S.A. would like to thank the University of Chicago and Yale University for their hospitality, and J.C. would like to thank the University of Illinois Urbana-Champaign and Yale University for their hospitality.  We would like to thank Mathematisches Forschungsinstitut Oberwolfach (MFO) and the organizers of the workshop ``Flat Surfaces, and Dynamics on Moduli Spaces", March 2014. We would like to thank the anonymous referee for their patient and careful reading which has vastly improved the clarity of the paper.

\section{Background material}\label{sec:background}

\subsection{Our spaces} This section recalls standard material which is treated in, for example, Zorich's survey \cite{Zorichsurvey}. 
A translation surface can be given by a union of polygons $P_1 \cup \dots \cup P_n$ where each $P_i \subset \C$, and that each side of each $P_i$ is glued to exactly one other (parallel) side by a translation, and the total resulting angle at each vertex is an integer multiple of $2 \pi$. Translation surfaces can be organized by the number and order of these singularities, that is by integer partitions $\alpha$ of $2g-2$, where $g$ is the genus of the surface, yielding \emph{strata} $\hh(\alpha)$. Since translations are holomorphic, and preserve the one-form $dz$, we obtain a complex structure and a holomorphic differential $\omega$ on the identified surface, which away from zeros is locally $dz$. The zeroes of the differential will be at the identified vertices with total angle greater than $2\pi$, and the order of the zero is equal to the excess angle, that is $\omega = z^{k} dz$ in a neighborhood of a point with total angle $2\pi(k+1)$. This paper concerns translation surfaces in the stratum $\hh(2)$, that is, genus $2$ surfaces with one singularity with angle $6 \pi$. Kontsevich-Zorich \cite{KZ} classified the connected components of strata. There are at most three and in our case there is only one, that is, $\hh(2)$ is connected.

By varying the sides of the polygons $P_i$ one changes the flat surface. This gives \emph{local coordinates} on strata (modeled on relative cohomology of the surface with respect to the singularities) which give the \emph{Masur-Veech} measure $\mu_{MV}$. $SL_2(\mathbb{R})$ acts on strata via linear action on the polygons $P_i$. Masur \cite{Masur1} and Veech \cite{Veech} showed that this action is ergodic with respect to $\mu_{MV}$ on connected components of the stratum.  On any translation surface, we have the straight line flow given by flow in the vertical direction in $\C$, and the flow in direction $\theta$, which is the vertical flow on the surface $e^{i\theta} \om$. For any translation surface, the first return map of the flow in a fixed direction to a transverse interval gives a special map of the interval, known as an \emph{interval exchange transformation}.

\begin{defin}Given $\lambda=(\lambda_1,\lambda_2,\ldots,\lambda_d)$
where $\lambda_i > 0$, we obtain $d$ sub-intervals of the
interval $[0,\underset{i=1}{\overset{d}{\sum}} \lambda_i)$: $$I_1=[0,\lambda_1) ,
I_2=[\lambda_1,\lambda_1+\lambda_2),\ldots,I_d=[\lambda_1+\ldots \lambda_{d-1}, \lambda_1+\ldots+\lambda_{d-1}+\lambda_d).$$ Given
 a permutation $\pi$ on  the set $\{1,2,\ldots,d\}$, we obtain a d-\emph{Interval Exchange Transformation} (IET)  $ T \colon [0,\underset{i=1}{\overset{d}{\sum}} \lambda_i) \to
 [0,\underset{i=1}{\overset{d}{\sum}} \lambda_i)$ which exchanges the intervals $I_i$ according to $\pi$. That is, if $x \in I_j$ then $$T(x)= x - \underset{k<j}{\sum} \lambda_k +\underset{\pi(k')<\pi(j)}{\sum} \lambda_{k'}.$$
 \end{defin}
\noindent For a small enough neighborhood in the space of translation surfaces $U$, one can locally fix a transversal where the IET has $2g+k-1$  intervals, where  $g$ is the genus  and $k$ is the number of singularities of the translation surface. This provides a  map $\T:U \rightarrow \R_+^{2g+k-1}$. This is a (locally) Lipshcitz map from $U$ with the metric given by coordinates to $\R_+^{2g+k-1}$ with the Euclidean metric. In fact, it is still locally Lipschitz if we compose it with the natural map $\lambda \mapsto \frac{\lambda}{|\lambda|}$, (where $|\lambda| = \sum \lambda_i$) to obtain a map from $U$ to the simplex $\Delta_{2g+k-1}: = \{\lambda \in \R_+^{2g+k-1}: \sum \lambda_i  = 1\}$.

\subsection{Rauzy induction}\label{sec:rauzy} The proof of our main result (and indeed many results on ergodicity of IETs) uses in a crucial fashion the Rauzy induction renormalization procedures for IETs, involving induced maps on certain subintervals, and closely related to Teichm\"uller geodesic flow. Our treatment of Rauzy induction will be the same as in \cite[Section 7]{Veech}. For further details of the procedure (and much more on IETs) we refer the interested reader, to, e.g.~\cite{Yoccoz}, for an excellent survey. 
 \begin{figure}[h!]\caption{The Rauzy Class of (4321). Dashed arrows represent `A'-moves, and solid `B'-moves .\medskip}\label{fig:rauzy}
\begin{tikzpicture}
\path[dotted,->]  (-.6,0) edge (-2.5, 1.7) 
node[left=1.4 cm, above=1.4 cm, text width= 3cm]{(2431)};
\path[dotted,->] (-2.9, 1. 4) edge (-2.9,.4);
\path[dotted,->] (-2.3,0) edge (-.9,0)
node[right=.35 cm, above=-.2 cm, text width=3cm]{(3241)};
\path[->] (-3.5,.17) edge [loop left] (-4,.15);
\path[->](-3.5,1.6) edge[bend left] (-5,1.6);
\path[->](-5,1.7) edge[bend left] (-3.5,1.7)
node[left= -.35 cm, above=-.35 cm, text width=3cm]{(2413)};
\path[dotted,->] (-6.2,1.6) edge[loop left](-6.2,1.6);
\path[->] (2,0) edge(.2,0);
\path[->](2.5,1.7) edge (2.5,.5)
(2.5,.5) node[ right= -1.9 cm, above=-.7 cm, text width=3 cm] {(4321)}
node[ right=1.2 cm, above=-.7 cm, text width=3cm]{(4213)}
node[ right=1.2 cm, above= 1.1 cm, text width=3cm]{(4132)};
\path[dotted,->] (3.2,.1) edge[loop right] (2.3,-.7);
\path[->](.2,.2) edge (2.2,1.9);
\path[dotted, ->] (3.2,2.0) edge[bend left] (4.5,2.0) 
node[ right=3 cm, above=-.3 cm, text width=3cm]{(3142)};
\path[dotted,->]  (4.5,1.9) edge[ bend left] (3.2,1.9);
\path[->](5.7,2.0) edge[loop right] (6,2);
\end{tikzpicture}
\end{figure}
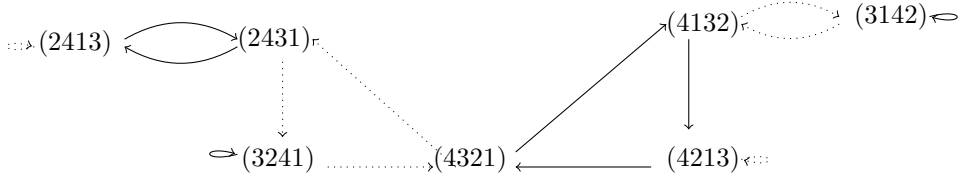

\begin{figure}[h!]\caption{Rauzy matrices for the edges in the Rauzy graph of $(4321)$. We divide into left and right sides and note the symmetries. \bigskip}\label{table:matrix}
\begin{tabular}{|c|c|c|c|}  \multicolumn{2}{c}{\huge Right Side}  & \multicolumn{2}{c}{\huge Left Side} \\\hline Edge & Matrix & Edge & Matrix \\\hline $(4321) \rightarrow (4132)$ & 
$\left(
\begin{array}{cccc}
 1& 0 &0  &0 \\
 0&  1&  0& 0 \\
 0& 0 & 1 & 0 \\
 1& 0 & 0 &  1\end{array} \right)$ &$(4321) \rightarrow (2431)$   & $ 
\left(
\begin{array}{cccc}
 1& 1 &0  &0 \\
 0&  0&  1& 0 \\
 0& 0 & 0 & 1 \\
 0& 1 & 0 &  0\end{array} \right)$ \\\hline $(4132) \rightarrow (4213)$ & $
\left(
\begin{array}{cccc}
 1& 0 &0  &0 \\
 0&  1&  0& 0 \\
 0& 0 & 1 & 0 \\
 0& 1 & 0 &  1\end{array} \right)$ & $(2431) \rightarrow (3241)$ & $ 
\left(
\begin{array}{cccc}
 1& 1&0  &0 \\
 0&  0&  1& 0 \\
 0& 0 & 0 & 1 \\
 0& 1 & 0 &  0\end{array} \right)$ \\\hline $(4213)\rightarrow(4321)$ & $ 
\left(
\begin{array}{cccc}
 1& 0 &0  &0 \\
 0&  1&  0& 0 \\
 0& 0 & 1 & 0 \\
 0& 0 & 1 &  1\end{array} \right)$ & $(3241) \rightarrow (4321)$ & $ 
\left(
\begin{array}{cccc}
 1& 1 &0  &0 \\
 0&  0&  1& 0 \\
 0& 0 & 0 & 1 \\
 0& 1 & 0 &  0\end{array} \right)$ \\\hline $(4132) \rightarrow (3142)$ & $ 
\left(
\begin{array}{cccc}
 1& 0 &0  &0 \\
 0&  1&  1& 0 \\
 0& 0 & 0 & 1 \\
 0& 0 & 1 &  0\end{array} \right)$ & $(2431) \rightarrow (2413)$ & $ 
\left(
\begin{array}{cccc}
 1& 0 &0  &0 \\
 0&  1&  0& 0 \\
 0& 0 & 1 & 0 \\
 1& 0 & 0 &  1\end{array} \right)$ \\\hline $(3142) \rightarrow (3142)$ & $ 
\left(
\begin{array}{cccc}
 1& 0 &0  &0 \\
 0&  1&  0& 0 \\
 0& 0 & 1 & 0 \\
 0& 1 & 0 &  1\end{array} \right)$ & $(2413) \rightarrow (2413)$ &  $ 
\left(
\begin{array}{cccc}
 1& 0 &0  &0 \\
 0&  1&  0& 0 \\
 0& 0 & 1 & 1 \\
 0& 0 & 0 &  1\end{array} \right)$ \\\hline $(3142)\rightarrow(4132)$ & $
\left(
\begin{array}{cccc}
 1& 0 &0  &0 \\
 0&  1&  1& 0 \\
 0& 0 & 0 & 1 \\
 0& 0 & 1 &  0\end{array} \right)$ & $(2413)\rightarrow(2431)$ & $
\left(
\begin{array}{cccc}
 1& 0 &0  &0 \\
 0&  1&  0& 0 \\
 0& 0 & 1 & 0 \\
 0& 0 & 1 &  1\end{array} \right)$ \\\hline $(4213) \rightarrow (4213)$ & $ 
\left(
\begin{array}{cccc}
 1& 0 &0  &0 \\
 0&  1&  0& 0 \\
 0& 0 & 1 & 1 \\
 0& 0 & 0 &  1\end{array} \right)$ & $(3241) \rightarrow (3241)$ & $
\left(
\begin{array}{cccc}
 1& 0 &0  &0 \\
 0&  1&  0& 0 \\
 0& 0 & 1 & 0 \\
 1& 0 & 0 &  1\end{array} \right)$ \\ \hline \end{tabular}
 \end{figure}
The Rauzy induction map $R$ is defined for all but a codimension $1$ set of IETs and associates to an interval exchange map $T= T_{\lambda, \pi}$, (now we restrict to $\lambda \in \Delta_m$, the unit simplex in $\R_+^m$, $\pi \in S_m$) a new interval exchange map $R(T) = T_{\lambda', \pi'}$, by considering the induced map of $T$ on the subinterval $[0, 1- \min(\lambda_{m}, \lambda_{\pi^{-1}m}))$, and renormalizing the lengths so $R(T)$ is again a map of $[0, 1)$.  $\lambda'$ is related to $\lambda$ via a projective linear transformation defined below. Rauzy induction is only defined if $\lambda_{m} \neq \lambda_{\pi^{-1}m}$. The \emph{Rauzy class} $\mathcal R$ of a permutation $\pi$ is the subset of $S_m$ that contains all the forward images of $\pi$ under Rauzy induction. The permutations in the Rauzy class form the vertices of the \emph{Rauzy graph}, a directed graph with two edges emanating from each permutation $\sigma \in \mathcal R$, corresponding to the permutations obtained by inducing on $[0, 1-\lambda_m)$ and $[0, 1-\lambda_{\sigma^{-1} m})$ respectively.

 If $\lambda_m = \min(\lambda_m, \lambda_{\pi^{-1}m})$ we say the first step in Rauzy induction is $A$. In this case the permutation of $R(T)$ is given by 
\begin{equation*} \pi'(j)= \begin{cases}
 \pi (j) & \quad j \leq \pi^{-1}(d)\\ \pi(d) & \quad j=\pi^{-1}(d)+1 \\ \pi(j-1) & \quad \text{otherwise}

\end{cases}.
\end{equation*}
We record the action of Rauzy induction by the elementary matrix $M(T,1)$ where 
\begin{equation*} M(T,1)[ij]= \begin{cases} \delta_{i,j} & \quad j \leq \pi^{-1}(d)\\
 \delta_{i, j-1} & \quad j>\pi^{-1}(d) \text{ and } i \neq d\\
\delta_{\pi^{-1}(d)+1,j} & \quad i=d \end{cases}. 
 \end{equation*}
 If $\lambda_{\pi^{-1}m} = \min(\lambda_m, \lambda_{\pi^{-1}m})$we say the first step in Rauzy induction is $B$.
 In this case the permutation of $R(T)$ is given by 
\begin{equation*} \pi'(j)= \begin{cases}
 \pi (j) & \quad \pi(j) \leq \pi(d)\\ \pi(j)+1 & \quad \pi(d) < \pi(j) < d \\ \pi(d)+1 & \quad \pi (j)=d

\end{cases}.
\end{equation*}
 We keep track of what has happened under Rauzy induction by a matrix \begin{equation*}M(T,1)[ij]= \begin{cases} 1 & \quad i=d \text{ and }j= \pi^{-1}(d) \\ \delta_{i,j} & \quad \text{ otherwise} \end{cases}. 
\end{equation*} 
We have $$\lambda = \frac{M(T, 1)\lambda'}{|M(T, 1)\lambda'|}.$$ Here, $| \cdot |$ denotes the $L_1$ norm on $\R^d$. $M(T, 1)$ depends on whether the step is $A$ or $B$ and the permutation $\pi$. We define the $n^{\text{th}}$ matrix of Rauzy induction by $$M(T,n)=M(T,n-1)M(R^{n-1}(T),1).$$ If $R^n(T) = \lambda^{(n)},$ we have $$\lambda =   \frac{M(T, n)\lambda^{(n)}}{|M(T, n)\lambda^{(n)}|}$$

%\begin{lem} \label{one step} If $T = T_{\lambda, \pi}$, $R(T)=T_{\lambda', \pi'}$  then the length vector $\lambda'$ is a scalar multiple of $M\lambda$, where $M$ is the matrix given by the first step of Rauzy induction from $T$. \end{lem}
%This is \cite[7.4 and 7.5]{gauss}.
\noindent Given a matrix $M$, we write $$M{\Delta}=M\mathbb{R}_d^+ \cap\Delta_d=\left\{\frac{Mv}{|Mv|}:v\in \Delta_d\right\}.$$ 

\begin{lem}\label{region} Let $T = T_{\lambda, \pi}$, $S = S_{\eta, \pi}$ be IETs. If $ \eta \in M(T,k){\Delta}$, then $$M(S, k) = M(T,k).$$ That is, the IETs $T$ and $S$ have the same first $k$ steps of Rauzy induction.
\end{lem}

\noindent We will be working with the Rauzy class of the permutation $(4321)$ on $4$ letters. We record the graph (Figure~\ref{fig:rauzy}) and the associated matrices (Figure~\ref{table:matrix}).

\section{Abstract Setup}\label{sec:frostman}
\begin{prop}\label{prop:abstract}Let $S_1\supset S_2 \supset ...$ be a nested sequence of finite unions of disjoint affine 3-simplices (in $\mathbb{R}^4$ or $\Delta_3$) so that 
\begin{enumerate}
\item There exists a constant $c>0$ so that any simplex in $S_k$ has 1 side of length at least $c$.
%\item The number of simplices in $S_k$ is...
\item There exists a constant $\rho>0$ and a quadratic polynomial $p(k)$ with leading coefficient $a>0$ so that each simplex in $S_k$ contains $\rho10^{p(k)}$ simplices in $S_{k+1}$.
\item There exists $r \in \mathbb{N}$ and $h(x)$, a cubic polynomial with leading coefficient $-b$ where $b>0$, so that  $z<10^{h(k)}$ and $p \in \mathbb{R}^4$ (or $\Delta_3$) then $B(p,z)$ intersects at most one simplex in $S_k$, $J$ so that $B(p,z) \cap S_{k+r}=B(p,z)\cap J \cap S_{k+r}$.
\end{enumerate}
Then $H_{dim}(\cap S_i)\geq 1+\frac{a}{3b}.$ In particular if $a=12$ and $b=\frac 8 3 $ then $H_{dim}(\cap S_i)\geq \frac 5 2$.
\end{prop}
To prove Proposition \ref{prop:abstract} we use Frostman's lemma:
\begin{thm}(Frostman's Lemma) Let $A$ be a Borel set, $s>0$, and $\mu$ be a measure on $A$ such that $$\mu(B(x,r))\leq Cr^s.$$ Then $$\Hdim(A)\geq s.$$
\end{thm}
\begin{proof}[Proof of Proposition \ref{prop:abstract}] First we build a sequence of measures on the $S_k$ whose weak-* limits will have that  $\cap_{k=1}^{\infty} S_k$ is a set of full measure. Let $\mu_1$ be defined to be the probability measure so that:
\begin{enumerate}
\item It gives equal mass to each element of $S_1$.
\item It is a scalar multiple of Lebesgue when restricted to any element of $S_1$.
\end{enumerate}
Given $\mu_{k-1}$ which is a probability measure which restricted to each $S_{k-1}$ is a scalar multiple of Lebesgue, we inductively define $\mu_k$ to be the probability measure so that on each element of $S_k$ it is a scalar multiple of Lebesgue satisfying the following: if $J,J'\subset I \in S_{k-1}$ then $\mu_k(J)=\mu_k(J')$. That is, we evenly divide the mass in $I$ to its descendants in $S_k$. Let $\mu_{\infty}$ be a weak-* limit of these measures (it is unique but this is not important for our purposes). 
\begin{lem}If $J \in S_k$ then $\mu_k(J)=\mu_L(J)$ for all $L\geq k$.
\begin{proof}We prove this by induction. It is clear when $L=k$. We now assume that it is true for $L=r \geq k$. Observe first that $\mu_{r+1}(J)\geq \mu_r(J)$ because $$ \mu_{r+1}(J)\geq \sum_{J \in S_{r+1}\cap J}\mu_{r+1}(J).$$ By construction of $\mu_{r+1}$ $$\sum_{J \in S_{r+1}\cap J}\mu_{r+1}(J)= \mu_r(J).$$ By disjointness of $J$ from the other elements of $S_k$ we have $$\mu_L(J)\leq \mu_k(J)$$ for all $L\geq k$.
\end{proof}
\begin{lem} For all $k\leq L$ we have that if $J,J'\in S_k$ then $\mu_L(J)=\mu_L(J')$.
\end{lem}
\end{lem}
\begin{proof} This follows by induction and the fact that $|S_{k+1}\cap J|=\rho10^{p(k)}$  for all $J \in S_k$.
\end{proof}
\begin{lem}$$\mu_L(B(x,r))\leq \frac{|\{J\in S_L: B(x,r)\cap J \neq \emptyset\}|}{|S_L|}g\frac{r}c$$ where $g$ depends only on dimension.
\end{lem}
\begin{proof}
If $J \in S_k$ consider slices of $J$ by parallel hyperplanes perpendicular to the long side of $J$. There exists a constant $e$ so that for a segment of the long side of length $\frac c e$ we have that the hyperplanes intersect $J$ in area at least $\frac 1 2 $ of the maximal area of such a hyperplane. Let $g= 2 e$.
\end{proof}
\begin{cor}Let $B(x,r) \cap S_{k+L} \subset J \in S_k$. Then 
$$\mu_{\infty}(B(x,r))\leq g\frac r c|S_k|^{-1}=g\frac r c(\prod_{i=1}^k(\rho 10^{p(k)}))^{-1}.$$
\end{cor}
A simple calculation shows that if $P$ a cubic polynomial with leading term $-\frac a 3$, then for every $\epsilon>0$ there exists a $C$ so that 
\begin{equation}\label{poly}10^{P(x)}<C(10^{h(k+1)})^{\frac a {3b}-\epsilon}.\end{equation}

\noindent We now complete the proof. For each $r$ let $k_r=\min\{L:r>10^{h(L)}\}$. By the previous corollary 
$\mu_{\infty}(B(x,r))\leq \frac{r}c|S_{k_r}|$. By Condition (2) this is at most $$\frac r c \rho^{k_r-1}10^{-\sum_{i=1}^{k_r-1}p(i)}=\frac r c \rho^k10^{-P(k_r)}$$ where $P(x)$ is a cubic polynomial with leading coefficient $\frac a 3$. It follows by our observation (\ref{poly}) above that for every $\epsilon$ there exists $C$ so that  $$\mu(B(x,r))<\frac r c C(10^{h(k_r-1)})^{\frac a {3b} -\epsilon}\leq \frac 1 cC r^{1+\frac a {3b} -\epsilon}.$$ 
\end{proof}

\begin{lem}\label{to verify} To verify Condition 3 of Propostion \ref{prop:abstract} it suffices to show that there exists $g$ a cubic polynomial with leading coefficient $-b$ so that the elements of $S_{k+2}$ avoid an $10^{g(k)}$ neighborhood of the boundary of $S_{k}$.
\end{lem}
\begin{proof}Let $h(k)=g(k)$. If $z<10^{h(k)}$ and  $B(p,z) $ intersects two elements of $S_{k}$ and so is in a $g(k)$ neighborhood of the boundary of an element of $S_k$. So $B(p,z) \cap S_{k+2}=\emptyset.$ Condition 3 follows with $r=2$.% then since $x<10^{g(k)}$ it can only intersect the element of $S_k$ containing it.
\end{proof}

\section{The paths we take}\label{sec:paths}

\noindent Rauzy induction provides a criterion (due to Veech) for non-unique ergodicity that is crucial for our construction. 
\begin{lem}\label{not ue}\cite[\S 1 and Proposition 3.22]{Veech78} Let $T$ be an IET so that $R^k(T)$ is defined for all $T$. If $T$ has exactly $r$ ergodic probability measures then $$\Delta_{\infty}(T) := \bigcap_{k=1}^{\infty} M(T,k)\Delta$$ is a subsimplex of dimension $r-1$. Every point in it gives an IET with $r$ ergodic probability measures, which are in bijective correspondence with the set of invariant measures of $T$.  \end{lem}

\noindent We will use this criterion to build a large set of IETs $T$ with at least $2$ invariant measures. For this, we need to consider some very specific paths. First, define the matrix $L_1(k)$ by going from (4321) to (4132) to (4213) and back to (4321) $n$ times. We have
$$ L_1(n) = \left(
\begin{array}{cccc}
 1& 0 &0  &0 \\
 0&  1&  0& 0 \\
 0& 0 & 1 & 0 \\
 n& n& n &  1\end{array} \right).$$
Similarly, define the matrix $U_1(n)$ by going from (4321) to (4132) to (4213) then looping at (4213) for $k$ times and then going back to (4321). We have
$$U_1(n) = \left(
\begin{array}{cccc}
 1& 0 &0  &0 \\
 0&  1&  0& 0 \\
 0& 0 & n+1 & n \\
 1& 1 & 1 &  1\end{array} \right).$$

 \noindent Given $A \in SL_2(\mathbb{Z}_+)$, we can write 
 \begin{equation}\label{sl2 decomposition} A=H_1^{p_1}H_2^{p_2}H_1^{p_3}\cdot \ldots\cdot H_1^{p_k}
 \end{equation}  for nonnegative integers $p_1,\ldots,p_r$ and where $$H_1=\left(\begin{array}{cc} 1 & 0 \\ 1 &1
\end{array}\right) \mbox{ and }H_2=\left(\begin{array}{cc} 1 & 1 \\ 0 &1
\end{array}\right).$$ Notice the interactions of the 3rd and 4th columns under $L_1(n)$ and $U_1(n)$ are $H_1^n$ and $H_2^nH_1$ respectively. This motivates the defintion $$N_1(A,r)=L_1^{p_1}U_1^{p_2}L_1^{p_3-1}U_1^{p_4}\ldots L_1^{p_k-1}U_1^r.$$

%\marginpar{Do we really want $p_i-1, i \geq 3$ here?}

\noindent Similarly we define $L_2(n),U_2(n)$ and $N_2(A,r)$ on the left hand side of the Rauzy graph.
 We will be especially concerned with $A$ and $r$ satisfying $$|A|\in I_k :=\left[10^{k^2-k},2\cdot 10^{k^2-k}\right] \mbox{ and } r \in J_k : = \left[10^{(k+1)^2+k},2\cdot 10^{(k+1)^2+k}\right].$$

\subsection{Properties}\label{sec:properties}
We consider matrices $M_k$ of the form
\begin{eqnarray*}\label{our matrices} M_k &=& M_k\left(\{A_i\}_{i=1}^{2k+2}, \{r_i\}_{i=1}^{2k+3}\right)\\  &=&N_1(A_{3},r_{4})N_2(A_{4},r_{5})\ldots N_1(A_{2i+1},r_{2i+2})N_2(A_{2i+2},r_{2i+3}) \ldots N_1(A_{2k+1},r_{2k+2})N_2(A_{2k+2},r_{2k+3})
\end{eqnarray*} where $|A_i|\in I_i$ and $r_i \in J_i$. Given a  matrix $M$, let $C_j(M)$ denote the $j^{th}$ column, $|C_j(M)|$ denote the sum of the entries in the $i^{th}$ column, and $C_{max}(M)$ denote the column with the largest sum of entries. 

\noindent\textbf{Notation} %It is convenient for our purposes to consider the action of matrices of Rauzy induction on the simplex.
% If $M$ is a matrix of Rauzy induction then let 
%$$M\Delta=M\mathbb{R}^4_+\cap \Delta=\{\frac{M\bar{v}}{|M\bar{v}|}:\bar{v}\in \Delta\}.$$
Recall that given a metric $d$ on a space $X$, we can define a pseudo-metric on subsets of $X$ via  $d(A,B)=\inf \{d(a,b);a\in A,b\in B\}$. We will use this where $d$ is the metric on the simplex $\Delta$ induced by angles between vectors. If $v, w$ are vectors let ${\spanD(v,w)=\{av+bw:a,b\geq 0\} \cap \Delta.}$ In general this can be empty but if $v,w \in \mathbb{R}^4_+$ it will not be. In the section below will view columns of matrices as elements of the simplex.
\subsection{What we will show}\label{sec:we show}
In this section we will prove that if $M_k$ has the form given above where the $A_i$ are all $D$-balanced $(D>9)$ for all $i\leq 2k+2$ then there exists $f$, a cubic polynomial with leading coefficient $-\frac 8 3$, quadratic polynomials $p,q$ and a cubic polynomial $H$ with leading coefficient $4$ so that 
\begin{enumerate}
\item $d\left(C_i(M_k),C_j(M_K)\right)> \frac 1 {900}$ for $i \in \{1,2\}$ and $j \in \{3,4\}$ (Proposition \ref{line seg}). 
\item $d \left(C_1(M_k),C_2(M_K) \right), d \left(C_3(M_k),C_4(M_k)\right) \in [10^{p(k)}10^{f(k)}, 10^{q(k)} 10^{f(k)}]$% where $p$ is a cubic polynomial with leading coefficient $\frac 1 3 $ and $q$ is a quadratic polynomial.
 (Proposition \ref{angle bound}).
%\item $d \left(C_3(M_k),C_4(M_k)\right) \in [\frac 1 {D'}^k10^{p(2k+2)}, 10^{q(k)}D'^k 10^{p(2k+2)}]$.
\item There are at least $\rho^kH(k)$ such $M_k$% where $H(k)$ is a cubic polynomial with leading coefficient $4$.
 (\S\ref{subsec:numb}).
\end{enumerate}
Note the connection between Conclusion 1 and Conclusion 1 of the Proposition \ref{prop:abstract}. Similarly for Conclusion 3 and Conclusion 2 of the Proposition \ref{prop:abstract}.
\subsection{Angle bounds}
\begin{lem}\label{column size} If $M_k$ is a matrix described above, then 
$$|C_j(M_q)|\in \left[ \prod_{i=4}^{2k+3} 10^{i^2},2^{2k}\cdot 2  \prod_{i=4}^{2k+3} 10^{i^2}\right]\text{ for }j=1,2$$ and
$$|C_j(M_q)|\in \left[\prod_{i=3}^{2k+2} 10^{i^2}, 2^{2k} \cdot 2 \prod_{i=3}^{2k+2} 10^{i^2}\right] \text{ for }j=3,4.$$
\end{lem}
\begin{proof}We have $$|C_j(MN_1(A,r))|\leq |C_j(M)|\cdot |A|\max\{|C_3(M)|,|C_4(M)|\}$$ for $j= 1,2.$ Similarly, $$|C_i(MN_1(A,r))|\geq |A|r \min\{|C_3(M)|,|C_4(M)|\}$$ for $j=3,4$. Similar inequalities hold for $N_2$. The lemma follows by induction with the extra factor of 2 absorbing $N_1$'s contribution to $C_1,C_2$ or $N_2$'s contribution to $C_3,C_4$.
\end{proof}

The remainder of this subsection is devoted to proving:

\begin{prop}\label{angle bound} Let $M_k$ be as given in Equation \ref{our matrices} where all of the $A_i$ are $D$-balanced. There exists a cubic polynomial $f$ with leading coefficient $-\frac 8 3$ and two quadratic polynomials $p,q$ so that  $d \left(C_1(M_k),C_2(M_K) \right)$ and  $d \left(C_3(M_k),C_4(M_k)\right)$ are  in $[10^{p(k)}10^{f(k)}, 10^{q(k)} 10^{f(k)}]$.
\end{prop}

We present an elementary lemma on how the angle between vectors changes under addition.

\begin{lem}\label{add vectors} Let $v, w \in \R^n,$ and let $\theta_0$ denote the angle between $v$ and $w$. If $\theta_1$ denotes the angle between $v+w$ and $w$, we have $$| \sin \theta_1 |= \frac{\|v\|}{\|v+w\|} |\sin \theta_0|.$$ In particular, if $v$ and $w$ are perpendicular, we have  $$ |\sin \theta_1 |= \frac{\|v\|}{\|v+w\|}.$$ \end{lem}
\begin{proof} If $v$ and $w$ are linearly dependent then both sides are zero. If not,  let $w'$ denote the vector $w$ rotated by $\pi/2$ in the plane spanned by $v, w$. Then $$|\sin \theta_0| = \frac{|\langle v, w' \rangle|}{\|v\|\|w\|},$$ and  $$|\sin \theta_1| = \frac{|\langle v+w, w' \rangle|}{\|v+w\|\|w\|} =   \frac{|\langle v, w' \rangle|}{\|v+w\|\|w\|},$$ proving the result.
\end{proof}
\begin{lem}\label{sl2 angle} Let $D>1$, and suppose $A = \left(\begin{array}{cc}a & b \\c & d\end{array}\right) \in SL_2(\mathbb{Z}_+)$ is $D$-balanced, that is, $$\frac 1 D \le \frac {\left|\left(\begin{array}{c}a \\c\end{array}\right)\right|}{\left|\left(\begin{array}{c}b \\d\end{array}\right)\right|} \le D.$$ Then if $\theta$ is the angle between $\left(\begin{array}{c}a \\c\end{array}\right)$ and $\left(\begin{array}{c}b \\d\end{array}\right)$, we have  $$\frac 1 {D'|A|^2}<|\theta|<\frac{D'}{|A|^2}.$$ $D'$ depends quadratically on $D$.
\end{lem}
\begin{proof} We have $$\sin \theta = \frac{ad-bc}{\left\|\left(\begin{array}{c}a \\c\end{array}\right)\right\| \left\|\left(\begin{array}{c}b \\d\end{array}\right)\right\|} = \frac{1}{\left\|\left(\begin{array}{c}a \\c\end{array}\right)\right\| \left\|\left(\begin{array}{c}b \\d\end{array}\right)\right\|}.$$
By the balancedness condition $\left\|\left(\begin{array}{c}a \\c\end{array}\right)\right\|, \left\|\left(\begin{array}{c}b \\d\end{array}\right)\right\|, |A|$ are all comparable up to a factor of $D$ and a universal factor ($\sqrt{2}$) from comparison of the $L^1$ and $L^2$ norms on $\R^2$. The other factor going into $D'$ is the Lipschitz constant of the $\arcsin$ function on $[-\frac 1 2 ,\frac 1 2]$.
\end{proof}

%\begin{prop}If $M_q$ is a matrix as above with the additional property that $A_i$ is $D$-balanced for all $i$ then 
%$$d(C_1(M_q),C_2(M_q))\in $$ $$d(C_3(M_q),C_4(M_q))\in$$
%\end{prop}
\begin{lem}\label{covering decay}
Let $A\in SL_2(\mathbb{Z}_+)$ be a $D$-balanced matrix. 
Then there exists a constant $D'$ depending only on $D$ such that $$\frac{d (C_3(M),C_4(M))}{D'|A|^2r}\leq d \left( C_3(MN_1(A,2r)),C_4(MN_1(A,r))\right) \leq D' \frac{d( C_3(M),C_4(M))}{|A|^2r}.$$ 
\end{lem}
\begin{proof} By the two previous lemmas there exists $D''$ depending only on $D$ such that
$$\frac{d (C_3(M),C_4(M))}{D''|A|^2}\leq d \left( C_3(MN_1(A,0)),C_4(MN_1(A,0))\right) \leq D'' \frac{d( C_3(M),C_4(M))}{|A|^2}.$$
 Notice that, for $J, K \geq 0$
 $$C_3(MN_1(A,K))=(K+1)C_3(MN_1(A,0))+C_4(MN_1(A,0))$$ and 
 $$C_4(MN_1(A,K))=KC_3(MN_1(A,0))+C_4(MN_1(A,0)).$$ So 
 $$C_3(MN_1(A,J))=C_4(MN_1(A,J))+(J+1)C_3(MN(A,0)).$$ The lemma follows by Lemma \ref{add vectors}.
\end{proof}
\begin{lem}\label{fixed r} Let $A\subset SL_2(\mathbb{Z}_+)$ be $D$-balanced. Then there exists $D'$ depending only on $D$ such that 
$$\frac{d (C_1(M),C_2(M))}{D'|A|^2r^2}\leq d \left( C_1(MN_2(A,r)),C_2(MN_2(A,r))\right) \leq D' \frac{d( C_1(M),C_2(M))}{|A|^2r^2}.$$
\end{lem}
This is similar to the proof of Lemma \ref{covering decay}.

\begin{proof}[Proof of Proposition \ref{angle bound}] This follows from the previous two lemmas and induction. Indeed, let $M_k$ and $M_{k+1}=M_kN_1(A_{2k+3},r_{2_k+4})N_2(A_{2k+4},r_{2k+4})$ satisfy our assumptions. By Lemma \ref{fixed r} 
$$\frac 1 {4D'} 10^{-((2k+4)^2+(2k+5)^2)}<d(C_3(M_{k+1}),C_4(M_{k+1}))<4D'10^{-((2k+4)^2+(2k+5)^2)}.$$
The leading term in the exponent is $-8k^2$. Summing these terms as $k$ varies gives a polynomial with leading term $-\frac 8 3 k^3$. We choose the quadratic polynomials $p(k)$ and $q(k)$ to absorb the possible lower order terms and factors of $D'$ and 2.

The argument for $C_1,C_2$ is similar.
\end{proof}

%\begin{cor} If $M_q$ is a matrix as above with the additional property that $A_i$ is $D$-balanced for all $i$ then 
%$$d(C_1(M_q),C_2(M_q))\in $$ $$d(C_3(M_q),C_4(M_q))\in$$
%\end{cor}
\subsection{Non-unique ergodicity}\label{sec:nue}
%We need one more technical lemma:
%\begin{lem}\label{abstract for prop}Let $U_k, V_k \subset \Delta$ be intervals in the simplex $\Delta$ so that 
%\begin{enumerate}
%\item$ \max\, \{d(y,U_k):y \in U_{k+1}\}>(1-s_k) d(U_k,V_k)$
%\item   $\max\, \{d(y,V_k):y \in V_{k+1}\}>(1-s_k) d(U_k,V_k)$.
%\end{enumerate}
%Let $U_{\infty}= \cap_{n=1}^{\infty} \overline{\cup_{k=n}^{\infty} U_k},$ 
%$V_{\infty}= \cap_{n=1}^{\infty} \overline{\cup_{k=n}^{\infty} V_k},$.
%If $\sum s_k<\infty$ then $d(V_{\infty},U_{\infty})>0$.
%\end{lem}
%\begin{proof}Choose $k$ so that $\sum_{i=k}^{\infty}s_k<\frac 1 9$. Observe that $\prod_{i=k}^{\infty} (1-s_i)\leq 1-\sum_{i=k}^{\infty} s_i$. 
%By the triangle inequality it follows that $d(U_{\infty},V_{\infty})<\frac 1 2 d(U_{k+1},J_{k+1})$.
%\end{proof}

\begin{prop}\label{line seg}Under our assumptions on $A_i$ and $r_i$, $$\bigcap_{k=1}^{\infty} M_k{\Delta}$$ is a non-degenerate line segment with length at least $\frac 1 {900}$.
%\end{prop}
%\begin{prop}
%If $\bar{p}$ is a path we take then there exists $c>0$ such that $$d(C_2(M(\bar{p},n)),C_3(M(\bar{p},n)))>c.$$
\end{prop}
Let  $$U_k= \spanD\left\{C_1\left(M_k\right), C_2 \left( N_1(A_{2i+1},r_{2i+2})N_2(A_{2i+2},r_{2i+3})\right) \right\}$$ and
$$V_k= \spanD\left\{C_1\left(M_k\right), C_2\left( N_1(A_{2i+1},r_{2i+2})N_2(A_{2i+2},r_{2i+3})\right)\right\}.$$
\begin{lem}$d(U_1,V_1)>\frac 1 {10}.$
\end{lem}
\begin{proof} Observe that if $$u\in \spanD \{C_1(N_1(A,r),C_2(N_1(A,r))\}$$ then the sum of the first and second coordinates of $u$ $$u_1 + u_2 = 1.$$ If $$v \in \spanD \{C_3(N_1(A,r),C_4(N_1(A,r))\}$$ then the first and second coordinates satisfy $$v_1 = v_2 =0.$$ Since $$\max\{|C_i(N_1(A,r)|\}_{i=1,2}\leq 2$$ it follows that $$d\left(\spanD \{C_1(N_1(A,r),C_2(N_1(A,r))\}, \spanD \{C_3(N_1(A,r),C_4(N_1(A,r))\}\right)\geq \frac 1 {2\sqrt{2}}.$$
We have $$\spanD \left\{C_1\left(N_1(A,r)N_2(A',r')\right),C_2\left(N_1(A,r)N_2(A',r')\right)\right\}\subset \spanD \{C_1(N_1(A,r)),C_2(N_1(A,r)))\}.$$
Moreover, if 
$$v\in  \spanD \{C_3(N_1(A,r)N_2(A',r')),C_4(N_1(A,r)N_2(A',r'))\}$$
 then $v=v'+u$ where 
$$|u|\leq \underset{i=1,2}{\max} \, C_i(N_1(A,r)N_2(A',0))|\leq 2\cdot 2\cdot 10^{4-2}$$ and $$|v|\geq\underset{i=3,4} {\min}\,|C_i(N_1(A,r))| \geq 10^4 \mbox{ and } v \in \spanD \{C_3(N_1(A,r),C_4(N_1(A,r))\}.$$ To see the condition on $u$ observe that the loop at $(3241)$ and arrow from $(3241)$ to $(4321)$ does not add anything to the columns that will be $C_3$ and $C_4$ at $(4321)$. By Lemma~\ref{add vectors}  
$$\max\{d(y,\spanD\{C_3(N_1(A,r),C_4(N_1(A,r))\}): y\in \spanD \{C_3(N_1(A,r),C_4(N_1(A,r))\}\}\leq \frac 1 {10}.$$ 
\end{proof}
\begin{lem}\label{each step}$d(U_{k+1},V_{k+1})>d(U_k,V_k)-2\cdot\frac {2^{2k}}{10^{2k}}$.
\end{lem}
\begin{proof}
By a similar argument to the second paragraph of the previous lemma 
\begin{multline}\label{eq:little move}\max \{d(y,U_{k+1}):y \in U_k\}\leq \\
\frac{\underset{i=3,4}{\max} |C_i \left(M_{k-1} N_1(A_{2i+3},r_{2i+4})N_2(A_{2i+4},0)\right)| }{\underset{i=1,2}{\min} |C_i \left( N_1(A_{2i+1},r_{2i+2})N_2(A_{2i+2},r_{2i+3})) N_1(A_{2i+3},r_{2i+4})\right)| }\leq %\frac{2\cdot %\max I_{2k+2}}{\min J_{2k+2}}
 \frac {2^{2k}}{10^{2k}}. 
\end{multline} 
The last inequality uses Lemma \ref{column size}.
Similarly $$\max\{d(y,V_{k+1}):y \in V_k\}\leq \frac {2^{2k}}{10^{2k}}.$$
Because $$d(U_{k+1},V_{k+1})\geq d(U_k,V_k)-(\max \{d(y,U_{k+1}):y \in U_k\}+\max \{d(y,V_{k+1}):y \in V_k\})$$ the lemma follows.
\end{proof}
\begin{proof}[Proof of Proposition \ref{line seg}] It is straightforward that the length of $\cap_{k=1}^{\infty}M_k\Delta$ is at least
 $\underset{k \to \infty}{\lim}\, d(U_k,V_k)$. By the previous two lemmas this is at least 
 $$\frac 1 {10}-2\sum_{k=2}^{\infty} \frac{2^{2k}}{10^{2k}}>\frac 1 {900}.$$
\end{proof}
\subsection{Number of described matrices}\label{subsec:numb}

\begin{prop}\label{count}There exists a cubic polynomial $H$ with leading coefficient 4 and a constant $\rho$ so that there are at least $\rho^k$ matrices $M_k$ satisfying conditions 1 and 2 of Section \ref{sec:we show}.
\end{prop}
\begin{lem} There exists $D>0$ and a polynomial $e$ of degree 2 and and leading coefficient 6 so that set of $D$-balanced matrices $N_i(A,L)$ with $|A|\in I_{2j+i} $ and $D$ balanced and $L \in J_{2j+i+1}$ is $10^{e(j)}.$
\end{lem}
\begin{proof} First observe that there exists $c>0$ so that the number of choices for $A$ is at least $c (10^{(2j)^2-2j})^2$. This follows because there number of positive matrices in $SL_2(\mathbb{Z}_+$ with norm between $R$ and $2R$ is proportional to $R^2$ and if $M \in SL_2(\mathbb{Z}^+)$ then 
$M\begin{pmatrix} 2 &1 \\ 1 &1 
\end{pmatrix}$ is 2-balanced. The lemma follows because there are at least $10^{(2j)^j+j}$ choices for $L$.
\end{proof}

\begin{proof}[Proof of Proposition \ref{count}] By applying the previous lemma to $N_1$ and $N_2$ we observe that there is quadratic polynomial $\tilde{e}$ with leading coefficient 12 so that for each matrix \begin{eqnarray*} M_k &=& M_k\left(\{A_i\}_{i=1}^{2k+2}, \{r_i\}_{i=1}^{2k+3}\right)\\  &=&N_1(A_{3},r_{4})N_2(A_{4},r_{5})\ldots N_1(A_{2i+1},r_{2i+2})N_2(A_{2i+2},r_{2i+3}) \ldots N_1(A_{2k+1},r_{2k+2})N_2(A_{2k+2},r_{2k+3})
\end{eqnarray*} 
where  the $|A_i| \in I_i$ and $D$ balanced and $r_i \in J_i$ there are at least $10^{\tilde{e}(k+1)}$ choices of $M_{k+1}$. So the total choices of $M_k$ is at least $\prod_{j=1}^k 10^{\tilde{e}(k)}$ which is greater than $10^{H(k)}$ for some cubic polynomial $H$ with leading coefficient 4. By the previous sections matrices of these forms satisfy conditions 1 and 2 of Section \ref{sec:we show}.
\end{proof}

%This follows because if $M \in SL_2(\mathbb{Z}^+)$ then 
%$M\begin{pmatrix} 2 &1 \\ 1 &1 
%\end{pmatrix}$ is 3-balanced. So to each matrix in $SL_2(\mathbb{Z})^+$ with norm between $[\frac N 2, \frac {2N}3]$ there exists a $3$-balanced matrix with norm between $N$ and $2N$. The number of matrices in $SL_2(\mathbb{Z})$ with norm between $R$ and $\frac 4 3 R$ is proportional to $R^2$.

\section{Verifying the assumptions of Proposition \ref{prop:abstract}}

We consider the simplices given in the previous section. They are in the 3-dimensional simplex, and are parallelapipeds. They have 4 long sides connecting $C_i$ and $C_j$ for $i \in \{1,2\}$ and $j\in \{3,4\}$. There are two short sides connecting $C_1,C_2$ and $C_3,C_4$. 

 They satisfy the conditions of Proposition \ref{prop:abstract} except Condition 3. 
 
 When $k=5$ we insert some additional conditions, in order to achieve the separation we require, which geometrically can be thought of as `chopping off the ends' of the long sides of the parallelapiped.
 \begin{enumerate}
 \item We delete a $10^{-5}$ neighborhood of $\spanD(C_1(M_5),C_2(M_5))$ and $\spanD(C_3(M_5),C_4(M_5))$.
 \item Inductively for each $k\geq 5$ consider $M_kN_1(A_{2k+3},r_{2k+4})N_2(A_{2k+4},r_{2k+5})$ given as in the previous section. We remove all of these simplices that contain the two $A_{2k+3}$ closest to the end points of $\spanD(C_(M_k),C_2(M_k))$ and the two $A_{2k+4}$ closest to the end points of $\spanD(C_3(M_k),C_4(M_k))$. 
 \end{enumerate}
 
\noindent Call the sets remaining after these deletions $S_k$. We claim that this sequence of sets satisfies Condition 3 for all $k\geq 5$. 
% Inductively given $M_k$ consider $M_kN_1(A_{2k+3},r_{2k+4})N_2(A_{2k+4},r_{2k+5})$ given as in the previous section. We remove all of these simplices that contain the two $A_{2k+3}$ closest to the end points of $\spanD(C_(M_k),C_2(M_k))$ and the two $A_{2k+4}$ closest to the end points of $\spanD(C_3(M_k),C_4(M_k))$. 
 
 %Next for $k>5$ remove a  $10^{-\frac 8 3k^3}$ neighborhood of $\spanD(C_1(M_k),C_2(M_k))$ and $\spanD(C_3(M_k),C_4(M_k))$. 
 
 \begin{lem} Let $k\geq 5$. Let $M \in S_k$ and $M'=MN_1(A_{2k+3},r_{2k+4})N_2(A_{2k+4},r_{2k+5})$%Let $k\geq 5$ and $M,M'$ be simplices in $S_k$ then
 \begin{multline}d\left(M'\Delta\cap S_{k+1},\partial M\Delta\right) \geq \frac 1 {10^7} \max\{ d\left(\spanD(C_1(M'),C_2(M')),\{C_1(M),C_2(M)\}\right),\\
  d\left(\spanD(C_3(M'),C_4(M')),\{C_3(M),C_4(M)\}\right)\}.
 \end{multline} 
 \end{lem}
 Here, we are measuring the separation of the line segments $$\spanD(C_i(M'),C_{i+1}(M')), i=1, 3$$ from the set of endpoints $\{C_i(M), C_{i+1}(M)\}$ of the previous line segments.
 %\begin{multline} d(M\Delta \cap S_k,M'\Delta \cap S_k)\geq \frac 1 {10^7} \max\{ d(\spanD(C_1(M),C_2(M)),\spanD(C_1(M'),C_2(M'))),\\
 % d(\spanD(C_3(M),C_4(M)),\spanD(C_3(M'),C_4(M')))\}.
 %\end{multline}

\begin{proof} This follows because we delete the neighborhood of the two sides, the simplices are convex combinations and Lemma \ref{each step}. 
\end{proof}

\begin{lem}$S_k$ satisfies the assumptions of Lemma \ref{to verify}.
\end{lem}
\begin{proof}By the previous lemma it suffices to bound $$\frac 1 {10^7} \max\{ d(\spanD(C_1(M'),C_2(M')),\{C_1(M),C_2(M)\},  d(\spanD(C_3(M'),C_4(M')),\{C_3(M),C_4(M)\})\}.$$  By our choice of deleting  two $A_{2k+3}$ closest to the end points of $\spanD(C_1(M_k),C_2(M_k))$ and Lemma \ref{fixed r} we have that $$d(\spanD(C_1(M'),C_2(M')),\{C_1(M),C_2(M)\}\geq d(C_1(M),C_2(M))\frac 1 {D'|A_{2k+4}|^2}.$$\end{proof} %This follows from the previous lemma and Lemmas \ref{covering decay} and \ref{fixed r}.

\noindent Thus, we may invoke Proposition \ref{prop:abstract} have shown the \emph{lower bound}: 
\begin{cor}\label{cor:lbound}The set of non-uniquely ergodic 4-IETs $\nue(4321) \subset \Delta$ satisfies $$\Hdim(\nue(4321)) \geq \frac 5 2 .$$
\end{cor}
\section{Bounds for Hausdorff dimension}\label{sec:lower}

\subsection{Technical Lemmas} Before we prove the upper bound for Hausdorff dimension, we require some technical lemmas:

\subsubsection{Lines and measures} In this subsection, we collect some technical lemmas on lines and measures in $\Delta$ and $\hh(2)$. First, we state a proposition relating Masur-Veech measure on the stratum $\hh(2)$ to the measure class on the \emph{set of line segments} in the simplex $\Delta_4$, which we view as $\Delta_4 \times \R^4$, and endow it with the Lebesgue measure class $m_{\ell} = m_{\Delta} \times m$, where $m_{\Delta}$ is the Lebesgue measure class on $\Delta_4$, and $m$ is the Lebesgue measure on $\R^4$ . 

Recall from \S\ref{sec:background} we have the map $\T: \hh(2) \rightarrow \Delta_4$, associating the normalized return map to the transversal to the vertical flow on $\omega$. Let $U \subset \hh(2)$, and consider the set of lines $$L(U) = \{\{\T(h_s \om)\}_{s \in [0, \epsilon]}: \om \in U, \epsilon >0\}$$ in $\Delta_4$ associated to horocycle trajectories based in $U$. Let $\mu_{MV}$ denote Masur-Veech measure on the stratum $\hh(2)$. It has been a long standing problem to understand the image of $L$, that is, to understand the set of line segments in $\Delta_4$ which arise as projections of horocycle trajectories. This problem was solved in full generality by Minsky-Weiss~\cite[\S5]{MW}, who showed~\cite[Theorem 5.3]{MW}, that given $(\mathbf{a}, \mathbf{b}) \in \Delta_4 \times \R^4$, the line segment defined by $$\{\mathbf{a} + s \mathbf{b}: s\in [0, 1]\}$$ is in the image of $L$ if and only if a certain quadratic form $Q$ evaluated at $\mathbf{a, b}$ is positive, and thus there is an open set of such pairs. Moreover, they show that for each pair $$(\mathbf{a_0,b_0}) \mbox{ with } Q(\mathbf{a_0, b_0)}>0,$$ there is an open neighborhood $U'$ of $(\mathbf{a_0, b_0})$ and a local \emph{affine} inverse map to the map $L$. Thus, we have:

\begin{prop} \label{lines} Let $U \subset \hh(2)$ be such that $L(U)$ is open. Suppose $A \subset L(U)$ with $m_{\ell} (A)>0$. Then  $\mu_{MV}(L^{-1}(A)) >0$. In particular, if $\mu_{MV}(U)>0$, then $m_{\ell}(L(U)) >0$.
\end{prop}

\noindent Second, we require the following:

\begin{lem}\label{open lines} If $U \subset \Delta_4 \times \R^4$ is an open set in the space of lines then there exists $V$ an open subset of $ \Delta_4$  and an open set of directions $\Theta \subset \R^4$ so that for every $\theta \in \Theta$ and $v \in V$ the line $$L(v, \theta): = \{v + t\theta: t \in (0, 1)\}$$ in direction $\theta$ through $v$ is in $U$.
\end{lem}
\begin{proof}  Let $\Theta$ be a set of directions so that $\bar{\Theta}$ is compact and for every $\theta \in \bar{\Theta}$ we have $L(v, \theta)\in U$. 
For each $\theta \in \bar{\Theta}$ there exists $\epsilon_{\theta}>0$ so that $d(w,v)<\epsilon_{\theta}$ we have $L(w,\theta)\in U$.
By compactness there exists $\epsilon>0$ so for all $\theta \in \bar{\Theta}$ and $w$ with $d(w,v)<\epsilon$ we have 
$L(w, \theta) \in U$. Let $V=B(v,\epsilon)$.
\end{proof}

\subsubsection{Metric Geometry} We also require some more general technical lemmas on Hausdorff dimension: Let
$\mathcal{H}^t$ be the $t-$dimensional Hausdorff measure. 
Let $C_t$ be the Riesz $t$-capacity (see \cite[Definition 8.4]{mattila} for the definition). %\marginpar{Reference to Hausdorff content?}.
 Let $Gras(n,n-m)$ denote the Grasmannian on $n-m$ planes in $n$ space. Let $\gamma_{n,n-m}$ denote the natural measure class on $Gras(n,n-m)$. If $W \in Gras(n,n-m)$ and $a \in W^{\perp}$ let $W_a$ be the translate of $W$ by $a$.
\begin{thm}(\cite[Theorem 10.8]{mattila})\label{mattila} Let $m\leq t\leq n$ and $A \subset \mathbb{R}^n$ with $C_t(A)>0$ then for $\gamma_{n,n-m}$ almost every $W \in Grass(n,n-m)$ 
$$\mathcal{H}^m(\{a\in W^{\perp}:C_{t-m}(A\cap W_a)>0\}).$$
\end{thm}
By \cite[Theorem 8.9 (3)]{mattila}, for Borel sets, capacity dimension and Hausdorff dimension are the same so we obtain the following corollary:
\begin{cor} If $A \subset \mathbb{R}^n$, $\Hdim(A)\geq t>n-1$ a positive measure set of lines in $\mathbb{R}^n$ intersect $A$ in a set with Hausdorff dimension at least $t$.
\end{cor}

\begin{lem}\label{open full dim} Let $V$ be a non-empty open subset of $\Delta_4$ and assume that the Hausdorff dimension of the set $\nue(4321)$ of not uniquely ergodic 4-IETs is at least $3-c$ as a subset of $\Delta_4$. Then $$\Hdim(\nue(4321)\cap V)\geq3-c .$$
\end{lem}
\begin{proof} There exists a matrix of Rauzy induction $M$ so that $M\Delta\subset V$ (pick a matrix so that the subsimplex $M\Delta$ is contained in a ball inside $V$). Being non-uniquely ergodic is Rauzy induction invariant so $M(\nue(4321)) \subset \nue(4321)$. $M$ is a bilipshitz map. (Note that the various $M$ are not uniformly bilipshitz but each individual one is bilipshitz.) Sincee $M(\nue(4321)) \subset \left(\nue(4321) \cap V\right)$, we have  $$\Hdim(\nue(4321)\cap V) \geq \Hdim (M(\nue(4321)))=\Hdim(\nue(4321))\geq 3-c.$$
\end{proof}

%
%%We first record a lemma on the size of the image of the map $L$, namely that the image of an open neighborhood $U$ covers most of the tangent space in $\Delta_4$ at points in the image of $T(U)$.
%%
%%\begin{lemma}\label{image} Let $U \subset \hh(2)$ be open, $\omega_0 \in U$ be such that the vertical foliation is uniquely ergodic. Then for all uniquely ergodic $\lambda \in \Delta_4$ sufficiently close to $\T(\omega_0)$ there exists a uniquely ergodic $\omega \in U$, $\epsilon >0$ such that $$\T(\omega') = \T(\omega_0)$$ and $$\T(h_{\epsilon} \omega') = \lambda.$$ \end{lemma}
%%
%%\begin{proof} By decomposing the surface $\omega_0$ into zippered rectangles, we can deform the heights of the rectangles by adding an  associated to $T(\omega_0)$ to produce a surface $\omega'$. Then, by adding a small $\epsilon$-multiple of the heights to the lengths of the intervals of $T(\omega_0)$, we can produce the interval exchange with lengths $\lambda$.
%%\end{proof}
%

%
%%\begin{lem}\label{open lines}Let $U$ be an open subset in the space of lines. Then there exists a non-empty open set $U'$ in the space of lines and a non-empty open set $V$ in $\Delta_3$ so that if $L \in U'$, $c \in \Delta_3$ and $L+c\cap V \neq \emptyset$ then $L+c\in U$.
%%\end{lem}
%
%%\begin{proof} We can take an open subset $U' \subset U$ with $\mu_{MV}(U') >0$. By the above discussion, $L(U')$ is an open subset of $\Delta_4 \times \Delta_4$, so we have that the $m_{\ell}(L(U'))>0$.
%%\end{proof}
\subsection{Upper Bound} In this section we prove:
\begin{thm} \label{thm:smaller} The set of minimal and not uniquely ergodic 4-IETs $\nue(4321) \subset \Delta_4$ has Hausdorff dimension at most $\frac 5 2 $, $$\Hdim(\nue(4321)) \le \frac 5 2.$$
\end{thm}
\noindent This result and Corollary~\ref{cor:lbound} establishes Theorem \ref{thm:iet}. We first note:
\begin{prop}\label{upper} Let  $c<1$ and suppose $h = \Hdim(\nue(4321)) > 3-c$. Then for  $\mu_{MV}$ almost every abelian differential $\omega$, $$\Hdim( \nue(\om)) \geq 1-c .$$\end{prop} 
\noindent
\begin{proof}[Proof of Proposition \ref{upper}] Let $U$ be an open set in the space of lines given by Minsky-Weiss \cite[Theorem 5.3]{MW}. Let $\Theta$ and $V$ be given by Lemma \ref{open lines}. By Theorem \ref{mattila} and Lemma \ref{open full dim} almost every element of $\{(v:\theta):v\in V, \theta \in \Theta\}$ has a translate that intersects $V\cap NUE$ in Hausdorff dimension at least $1-c$. By our assumption (Lemma \ref{open lines}) these lines are in $U$. The set of flat surfaces they come from has positive $\mu_{MV}$ measure.  By $SL_2(\mathbb{R})$-invariance of Hausdorff dimension of non-uniquely ergodic directions and the ergodicity of $\mu_{MV}$ with respect to the $SL_2(\mathbb{R})$ action, this set must have full measure. 
\end{proof}
%Consider the set of lines in $\Delta_3$. By Theorem \ref{mattila} a positive measure set of these intersect NUE in Hausdorff dimension $\alpha$. By Proposition~\ref{lines} these come from a positive measure set of flat surfaces

\noindent Recall Masur's upper bound:
\begin{thm}\label{thm:mbound}~\cite[Main Theorem]{Mhdim} For every abelian differential $\omega$, the Hausdorff dimension of the set of not uniquely ergodic directions is at most $\frac 1 2$.
\end{thm}

\begin{proof}[Proof of Theorem \ref{thm:smaller}] By Masur's upper bound, $$\Hdim( \nue(\om)) \le 1/2,$$ and thus, by Proposition~\ref{upper}, $h \le 5/2$.
\end{proof}

\subsection{Corollaries}\label{sec:cor} Finally, we prove Theorem~\ref{thm:stratum} and Theorem~\ref{thm:main} using Theorem~\ref{thm:iet}. The key result is the fundamental property of Hausdorff dimension: $$\Hdim (X \times Y) \geq \Hdim (X) + \Hdim (Y).$$
\subsubsection{Proof of Theorem~\ref{thm:stratum}} $\hh(2)$ is a fiber bundle over $\Delta_4$ via the map $\T$, and the vertical foliation of $\omega$ is non-uniquely ergodic if and only if $\T(\omega)$ is a non-uniquely ergodic IET. The fibers of $\T$ are $4$-dimensional (and in fact $\T$) defines a local product structure on $\hh$. Thus we have the Hausdorff dimension of the set of $\omega$ with a non-uniquely ergodic vertical foliation is given by $$4+ \frac5 2 = 7 - \frac 1 2$$ as claimed.\qed

\subsubsection{Proof of Theorem~\ref{thm:main}} Combining Theorem~\ref{thm:mbound} for the upper bound and Corollary~\ref{cor:lbound} with Proposition~\ref{upper} for the lower bound, we have that the set of flat surfaces $\om$ satisfying
%By the ergodicity of the $SL(2, \R)$-action on $(\hh(2), \mu_{MV})$, the function $$f(\omega) = \Hdim\{ \theta: \mbox{ vertical foliation of } e^{i\theta}\omega \mbox{ is non-uniquely ergodic}\}$$ is constant $\mu_{MV}$-almost everywhere on $\hh(2)$. Thus, since the space $\hh(2)/SO(2)$ is $6$-dimensional, and $f$ descends to a function on $\hh(2)/SO(2)$, we must have, by Theorem~\ref{theorem:foliation}, $f(\omega) = 1/2$ for $\mu_{MV}$-almost every $\omega$.\qed
$$\Hdim( \nue(\om)) =\frac 1 2  $$
 has positive $\mu_{MV}$-measure. By $\mu_{MV}$-ergodicity of $SL_2(\mathbb{R})$, it must have full $\mu_{MV}$-measure. A similar argument using only the upper bound yields Theorem~\ref{ub}.


\begin{thebibliography}{99}
\bibitem{yit thesis} Y. ~Cheung \textit{Hausdorff dimension of the set of nonergodic directions. With an appendix by M. Boshernitzan.} Ann. of Math. (2) 158 (2003), no. 2, 661-678.
\bibitem{CHM}Y.~Cheung, P.~Hubert and H.~Masur, \textit{Dichotomy for the Hausdorff dimension of the set of nonergodic directions.} Invent. Math. 183 (2011), no. 2, 337-383.


\bibitem{Keane} M.~Keane, \textit{Interval Exchange Trasformations}, Mathematische Zeitschrift, v. 141,  25-31, 1975.


\bibitem{Keane2} M.~Keane, \textit{Non-ergodic interval exchange transformations}
Israel J. Math. 26 (1977), no. 2, 188-196. 

\bibitem{KMS} S.~Kerckhoff, H.~Masur and J.~Smillie, \textit{Ergodicity of billiard flows and quadratic differentials.} {Ann. of Math. (2)}  124  (1986),  no. 2, 293-311.


\bibitem{KeynesNewton} H.~B.~Keynes and D.~Newton, \textit{ A Minimal, Non-Uniquely Ergodic Interval Exchange Transformation}, Math. Z. 148 (1976) 101-105. 

\bibitem{KZ} M.~Kontsevich and A.~Zorich, \emph{Connected components of
the moduli spaces of Abelian differentials with prescribed
singularities}, Invent. Math.  v. 153  (2003),  no. 3, 631-678.

\bibitem{Masur1} H.~Masur, \textit{Interval exchange transformations and measured foliations}, Annals of Mathematics, v. 115, 169-200, 1982.

\bibitem{Mhdim}  H.~Masur, \emph{Hausdorff dimension of the set of nonergodic foliations of a quadratic differential}, Duke Math. J., v. 66, 387-442, 1992.


\bibitem{MS} H.~Masur and J.~Smillie, \textit{Hausdorff Dimension of Sets of Nonergodic Measured Foliations}, Annals of Mathematics, Second Series, Vol. 134, No. 3 (Nov., 1991), pp. 455-543.

\bibitem{mattila} P.~Mattila \textit{Geometry of sets and measures in Euclidean spaces. Fractals and rectifiability.} Cambridge Studies in Advanced Mathematics, 44. Cambridge University Press, Cambridge, 1995. xii+343 

\bibitem{MinskyWeiss} Y.~Minsky, and B.~Weiss, \emph{Non-divergence of
horocyclic flows on moduli space},  J. Reine Angew. Math. v. 552, 131-177, 2002.

\bibitem{MW} Y.~Minsky, and B.~Weiss, \emph{Cohomology classes represented by measured foliations, and Mahler's question for interval exchanges}, to appear, Annales scientifiques de l'ENS.

\bibitem{Sataev} E.~ A.~ Sataev, \emph{The number of invariant measures for flows on orientable surfaces}  Izv. Akad. Nauk SSSR Ser. Mat. 39 (1975), no. 4, 860-878.

\bibitem{Veech69} W.~Veech, \textit{A Kronecker-Weyl theorem modulo 2.}
Proc. Nat. Acad. Sci. U.S.A. 60 1968 1163-1164.

\bibitem{Veech78} W.~Veech, \textit{Interval exchange transformations.} J. Analyse Math. 33 (1978), 222-272.

\bibitem{Veech}W.~Veech, \textit{Gauss measures for transformations on the space of interval exchange maps}, Annals of Mathematics, v. 115, 201-242, 1982.

\bibitem{Yoccoz} J.-C.~Yoccoz, \textit{Continued Fraction Algorithms for Interval 
Exchange Maps: an Introduction},  in ``Frontiers in Number Theory, Geometry and Physics", proceedings 
of the Spring School at Les Houches, France, March 2003 

\bibitem{Zorichsurvey} A.~Zorich, \textit{Flat surfaces}, Frontiers in number theory, physics, and geometry. I, 437-583, Springer, Berlin, 2006. 

\end{thebibliography}
\end{document}